\documentclass[11pt]{amsart}
\usepackage{amsmath}
\usepackage{amsthm}
\usepackage{amscd}
\usepackage{amssymb}
\usepackage{amsfonts}
\usepackage{graphics}
\usepackage{color}
\usepackage{comment}
\usepackage[dvipsnames]{xcolor}
\usepackage[all,cmtip]{xy}
\usepackage[normalem]{ulem}
\date{}

\newlength{\defbaselineskip }
 \long\def\salta#1{\relax}

 \theoremstyle{plain}
\newtheorem{theorem}{Theorem}[section]

\newtheorem{lemma}[theorem]{Lemma}

\newtheorem{corollary}[theorem]{Corollary}

\theoremstyle{definition}

\newtheorem{remark}[theorem]{Remark}

\newcommand{\R}{\mathbb{R}}
\newcommand{\N}{\mathbb{N}}
\newcommand{\Sm}{\mathbb{S}}

\numberwithin{equation}{section}

\title[An inequality by Del Pino-Dolbeault and the log-MT inequality]{
On the equivalence between an Onofri-type inequality by Del Pino-Dolbeault and the sharp logarithmic Moser-Trudinger inequality}
\author[N. Borgia]{Natalino Borgia}
\address{Dipartimento di  Matematica  \\ Universit\`{a} degli Studi di Bari Aldo Moro \\ Via Orabona 4\\ 70125 Bari, Italy}
\email{natalino.borgia@uniba.it}
\author[S. Cingolani]{Silvia Cingolani}
\address{Dipartimento di  Matematica  \\ Universit\`{a} degli Studi di Bari Aldo Moro \\ Via Orabona 4\\ 70125 Bari, Italy}
\email{silvia.cingolani@uniba.it}
\author[G. Mancini]{Gabriele Mancini}
\address{Dipartimento di  Matematica  \\ Universit\`{a} degli Studi di Bari Aldo Moro \\ Via Orabona 4\\ 70125 Bari, Italy}
\email{gabriele.mancini@uniba.it}

\begin{document}            

\begin{abstract}
In this paper we consider the $N$-dimensional Euclidean Onofri inequality proved by del Pino and Dolbeault in \cite{DD3} for smooth compactly supported functions in $\R^N$, $N \geq 2$.
We extend the inequality to a suitable weighted Sobolev space, although no clear connection with standard Sobolev spaces on $\mathbb{S}^N$ through stereographic projection is present, except for the planar case.
Moreover, in any dimension $N \geq 2$, we show that the Euclidean Onofri inequality is equivalent to  the logarithmic Moser-Trudinger inequality with sharp constant proved by Carleson and Chang in \cite{CC} for balls in $\R^N$. 
\end{abstract}



\date{\today}

\maketitle

\section{Introduction}
 In this paper we discuss connections between Onofri-type inequalities on the Euclidean space and the logarithmic Moser-Trudinger inequality on balls in $\R^N$.  Let $N \ge 2$ be an integer and let $\omega_{N-1}$ and $V_N$ denote respectively the $N-1$ dimensional measure of the unit sphere and the volume of the unit ball in $\R^N$. In \cite{DD3} del Pino and Dolbeault obtained the following $N$-dimensional Euclidean Onofri inequality for any $u \in C_0^{\infty}(\mathbb{R}^N)$:
\begin{equation}\label{EuclideanOnofriNdimensional}
\displaystyle  \ln \left( \int_{\mathbb{R}^N} e^u \, d\mu_N \right) \leq  \frac{1}{\widetilde{\omega_N}} \int_{\mathbb{R}^N} H_N(u,\mu_N) \, dx + \int_{\mathbb{R}^N} u \, d\mu_N,
\end{equation}
where
\begin{equation}\label{Not1}
\widetilde{\omega_N}:= N^N \left( \frac{N}{N-1} \right)^{N-1} \omega_{N-1},
 \end{equation}
\begin{equation}\label{Not2} \mu_N(x):= \frac{1}{V_N \left( 1+ |x|^{\frac{N}{N-1}} \right)^N}, \quad  d \mu_N(x):=\mu_N(x) \, dx,
\end{equation}
and
\begin{equation}\label{Not3} H_N(u,\mu_N) :=| \nabla v_N + \nabla u |^N - |\nabla v_N|^N - N |\nabla v_N |^{N-2} \nabla v_N \cdot \nabla u, \end{equation}
with 
\begin{equation}\label{Not}
v_N(x):=\ln \, \mu_N(x).
\end{equation}

They were able to obtain \eqref{EuclideanOnofriNdimensional}  by considering the endpoint of a family of optimal Gagliardo–Nirenberg interpolation inequalities, discovered in \cite{DD1} and extended in \cite{DD2}.
We point out that the Euclidean Onofri inequality \eqref{EuclideanOnofriNdimensional} was obtained with different techniques also by Agueh, Boroushaki, and Ghoussoub in \cite{ABG}, for any 
$u \in W^{1,N}(\mathbb{R}^N)$. We also refer to \cite{LL} for an improved version of \eqref{EuclideanOnofriNdimensional} involving singular weights (see also \cite{DEJ,DET,DLL}).  

In dimension $N=2$, \eqref{EuclideanOnofriNdimensional} becomes
\begin{equation}\label{EuclideanOnofri2D}
\displaystyle  \ln \left(\int_{\mathbb{R}^2} e^u \, d \mu_2 \right) \leq  \frac{1}{16 \pi} \int_{\mathbb{R}^2} | \nabla u |^2 \, dx + \int_{\mathbb{R}^2} u \, d \mu_2
\end{equation}
for any $u \in C_0^{\infty}(\mathbb{R}^N)$. Actually, the 2-dimensional Euclidean Onofri inequality \eqref{EuclideanOnofri2D} holds for any function $u \in L^1(\mathbb{R}^2, d \mu_2)$ such that $ |\nabla u| \in L^2(\mathbb{R}^2, dx)$. Indeed, by using the stereographic projection from $\mathbb{S}^2$ onto $\mathbb{R}^2$ with respect to the north pole, \eqref{EuclideanOnofri2D} is equivalent to the following classical Onofri inequality
\begin{equation}\label{LOMTS2}
 \ln \left( \frac{1}{4 \pi} \int_{\mathbb{S}^2} e^u \, dv_{g_0} \right) \leq  \frac{1}{16 \pi} \int_{\mathbb{S}^2} | \nabla_{g_0} u |_{g_0}^2 \, dv_{g_0} + \frac{1}{4 \pi} \int_{\mathbb{S}^2}u \, dv_{g_0},
\end{equation}
for any $\displaystyle u \in H^{1}(\mathbb{S}^2,g_0)$, where $g_0$ is the standard metric on $\mathbb{S}^2$. Inequality \eqref{LOMTS2} plays an important role in spectral analysis for the Laplace-Beltrami operator thanks  to Polyakov's formula (see \cite{OPS,OPS2,Poly1,Poly2}) and it was proved in 1982 by Onofri in \cite{O}  using conformal invariance and an earlier result by Aubin  \cite{Aubin}. 

Inequality \eqref{LOMTS2} is also strictly related to the Moser-Trudinger inequality \cite{M,poho,Tru,Yudo}. Specifically, in \cite[Theorem 2]{M} Moser proved via symmetrization techniques that there exists a constant $S>1$ such that 
\begin{equation}\label{MTS2}
\frac{1}{4\pi}\int_{\Sm^2} e^{4\pi \left( \frac{u - \frac{1}{4\pi}\int_{\Sm^2} u \, dv_{g_0}}{\int_{\Sm^2}|\nabla_{g_0}u|_{g_0}^2\,dv_{g_0}}\right)^2}\,dv_{g_0}\le S,
\end{equation}
for any $u\in H^1(\Sm^2,g_0)$. As a straightforward consequence of \eqref{MTS2} and Young's inequality, one can prove that
\begin{equation}\label{NOptOnofri}
 \ln \left( \frac{1}{4 \pi} \int_{\mathbb{S}^2} e^u \, dv_{g_0} \right) \leq  \frac{1}{16 \pi} \int_{\mathbb{S}^2} | \nabla_{g_0} u |_{g_0}^2 \, dv_{g_0} + \frac{1}{4 \pi} \int_{\mathbb{S}^2}u \, dv_{g_0}+\ln S,
\end{equation}
which is a non-optimal version of \eqref{LOMTS2} due to the constant $\ln S>0$.  For this reason \eqref{LOMTS2} is sometimes called Moser-Onofri or Moser-Trudinger-Onofri inequality in the literature. 

 An extension of \eqref{LOMTS2} to $\mathbb{S}^N$ in higher dimensions was stated in \cite{Bek,carlenloss}, however no clear connection with \eqref{EuclideanOnofriNdimensional} through stereographic projection is present.
Taking into account this, we want to extend \eqref{EuclideanOnofriNdimensional} to a suitable weighted Sobolev space. More precisely, we will consider the space

\begin{align*}
\displaystyle 
W_{\mu_N}(\mathbb{R}^N):= \{ u \in L^1(\mathbb{R}^N,d\mu_N): \quad  & |\nabla u| \in L^N(\mathbb{R}^N,dx),\\                        & |\nabla u|^2 |\nabla v_N|^{N-2} \in L^1(\mathbb{R}^N,dx) \}, 
\end{align*}

\noindent endowed with the norm
\begin{equation}\label{norm}
\lVert u \rVert_{\mu_N}:= \int_{\mathbb{R}^N} |u| \, d\mu_N + \lVert \nabla u \rVert_N + \left( \int_{\mathbb{R}^N}  |\nabla u|^2 |\nabla v_N|^{N-2} \, dx \right)^{\frac{1}{2}}.
\end{equation}

We observe that, if $N=2$, $W_{\mu_2}(\mathbb{R}^2)$ coincides with the set of functions $u \in L^1(\mathbb{R}^2, d \mu_2)$ such that $| \nabla u| \in L^2(\mathbb{R}^2, dx)$, in which \eqref{EuclideanOnofri2D} holds.

To extend \eqref{EuclideanOnofriNdimensional} to $W_{\mu_N}(\mathbb{R}^N)$, we show that smooth compactly supported functions of $\mathbb{R}^N$ are dense in $W_{\mu_N}(\mathbb{R}^N)$ with respect to the norm   \eqref{norm}. 

\begin{theorem}\label{densita} Assume $N \ge 2$. Then 
$$ \displaystyle W_{\mu_N}(\mathbb{R}^N)=\overline{C_0^{\infty} (\mathbb{R}^N)}^{\lVert \cdot \rVert_{\mu_N}}.$$
\end{theorem}

Therefore, by using the previous density result, we are able to prove the following Theorem.

\begin{theorem}\label{EuclideanOnofriNdim}
Assume $N \geq 2.$ Then, for any $u \in W_{\mu_N}(\mathbb{R}^N)$, we have 
\begin{equation}\label{OnofriEuclideaEstesa}
\displaystyle  \ln \left( \int_{\mathbb{R}^N} e^u \, d\mu_N \right) \leq  \frac{1}{\widetilde{\omega_N}} \int_{\mathbb{R}^N} H_N(u,\mu_N) \, dx + \int_{\mathbb{R}^N} u \, d\mu_N.
\end{equation}
\end{theorem}

\medskip
Now, let us recall the Logarithmic Moser-Trudinger inequality on domains of $\mathbb{R}^N$.  Let $\Omega \subset \mathbb{R}^N$ be a bounded domain. Similarly to \eqref{NOptOnofri}, by means of the Moser-Trudinger inequality for the Sobolev space $W^{1,N}_0(\Omega)$ \cite[Theorem 1]{M}, it is simple to prove that there exists a constant $c=c(N)$ such that 
\begin{equation}\label{LMTRnOmega} 
\ln \left( \frac{1}{|\Omega|} \int_{\Omega} e^u \, dx \right) \leq  \frac{1}{\widetilde{\omega_N}} \int_{\Omega} | \nabla u |^N \, dx + \ln c,
\end{equation}
for any $ \displaystyle u \in W^{1,N}_0(\Omega)$.

\medskip
In \cite{CC}, Carleson and Chang improved inequality \eqref{LMTRnOmega} when $\Omega$ is the unit ball $B_1$ of $\R^N$.  Indeed they prove that for any $u\in W^{1,N}_0(B_1)$: 
\begin{equation}\label{CarlesonChang}
\displaystyle \ln \left( \frac{1}{V_N} \int_{B_1} e^u \, dx \right) <  \frac{1}{\widetilde{\omega_N}} \int_{B_1} | \nabla u |^N \, dx + \sum_{k=1}^{N-1} \frac{1}{k}.
\end{equation}
Actually, by \eqref{CarlesonChang} one can easily show that for any open ball $B \subset \mathbb{R}^N$, and for any $u  \in W^{1,N}_0(B)$, we have  
\begin{equation*}
\displaystyle   \ln \left( \frac{1}{|B|} \int_{B} e^u \, dx \right) <  \frac{1}{\widetilde{\omega_N}} \int_{B} | \nabla u |^N \, dx + \sum_{k=1}^{N-1} \frac{1}{k}.
\end{equation*}

\noindent
Inequality \eqref{CarlesonChang} is sharp, in the sense that the functional $J:W^{1,N}_{0}(B_1) \to \mathbb{R}$ defined as 
\begin{equation}\label{Jfunctional}
\displaystyle J(u) := \frac{1}{\widetilde{\omega_N}} \int_{B_1} | \nabla u |^N \, dx  - \ln \left( \frac{1}{V_N} \int_{B_1} e^u \, dx \right),
\end{equation}
satisfies 
$$
\inf_{W^{1,N}_{0}(B_1)} J = - \sum_{k=1}^{N-1} \frac{1}{k}.
$$

Our goal is to establish a natural connection between the sharp inequality \eqref{CarlesonChang} and the inequality \eqref{OnofriEuclideaEstesa}. 
We can consider the functional $I:W_{\mu_N}(\R^N) \to \mathbb{R}$ defined as
\begin{equation}\label{Ifunctional}
I(u):=   \frac{1}{\widetilde{\omega_N}} \int_{\mathbb{R}^N} H_N(u,\mu_N) \, dx + \int_{\mathbb{R}^N} u \, d\mu_N- \ln \left(\int_{\mathbb{R}^N} e^u \, d\mu_N \right). 
\end{equation}

In our main result,  without using \eqref{EuclideanOnofriNdimensional} or \eqref{CarlesonChang}, we prove the following identity between $\displaystyle{\inf_{W^{1,N}_0(B_1)} I}$ and $\displaystyle{\inf_{W_{\mu_N}(\R^N)} J}$.

\begin{theorem}\label{equivalence} For any $N\ge 2$, we have 
$$
\inf_{W_{\mu_N}(\R^N)} I = \inf_{W^{1,N}_0(B_1)}  J +\sum_{k=1}^{N-1}\frac{1}{k}.
$$
\end{theorem}
Theorem \ref{equivalence} shows that the sharp inequalities \eqref{OnofriEuclideaEstesa} and \eqref{CarlesonChang} are equivalent. In particular, we can give a simple alternative proof of \eqref{OnofriEuclideaEstesa}. On the one hand, it follows  from \eqref{CarlesonChang} that, for any $u\in W^{1,N}_0(B_1)$:  
$$
J(u) > -\sum_{k=1}^{N-1} \frac{1}{k} = \inf_{W^{1,N}_0(B_1)} J.
$$
On the other hand, since $I(0)=0$, Theorem \ref{equivalence} yields
$$
0 = \inf_{W_{\mu_N}(\R^N)} I = \min_{W_{\mu_N} (\R^N)} I.  
$$

\noindent

 Theorem \ref{equivalence} is interesting even in the simplest case $N=2$. Indeed, as  announced in \cite{BCM} (see also \cite{IM}),  it shows that Onofri's inequality \eqref{LOMTS2} on $\Sm^2$ is equivalent to the sharp Logarithmic Moser-Trudinger inequality \eqref{CarlesonChang} by Carleson and Chang. 
 
 As discussed above  the inequalities with non-optimal constants \eqref{NOptOnofri} and \eqref{LMTRnOmega}  are  direct corollaries of the Moser-Trudinger inequality for the  sphere  \cite[Theorem 2]{M} and its analogue for bounded domains \cite[Theorem 1]{M}. In \cite{M}, Moser proved these theorems indipendently,
although their proofs are based on similar arguments relying on symmetrization techniques. He explicitly states that it seems impossible to deduce one Theorem from the other. Although the equivalence between Moser's theorems  is still an open problem, here we have proved  the equivalence between the corresponding  logarithmic inequalities with sharp constants \eqref{CarlesonChang} (with $N=2$) and  \eqref{LOMTS2}. 

\section{Preliminaries and function spaces}

\medskip
In this section we establish some preliminary estimates and discuss the optimal weighted Sobolev  space in which the $N$-Euclidean Onofri inequality holds. Throughout the paper we let $\widetilde \omega_N, \mu_N, d\mu_N$ and $v_N$ be defined as in \eqref{Not1}, \eqref{Not2}, \eqref{Not3} and \eqref{Not}.
A straightforward computation in polar coordinates shows that $d\mu_N$ is a probability measure on $\mathbb{R}^N$, namely
\begin{align*}
\displaystyle
     \int_{\mathbb{R}^N} \, d \mu_N =  1.
\end{align*}
Moreover, we have
$$ \displaystyle \nabla v_N(x)=-\frac{N^2}{N-1} \, \frac{|x|^{\frac{1}{N-1}}}{1+|x|^{\frac{N}{N-1}} } \,\frac{x}{|x|},$$
and
\begin{equation}\label{GaussGreen}
 \displaystyle  \Delta_N v_N(x) =- N^N \left( \frac{N}{N-1} \right)^{N-1} V_N \mu_N(x),
\end{equation}
where $\Delta_N$ denotes the  $N$-Laplacian operator, defined by

$$ \Delta_N u := \text{div}\left( | \nabla u |^{N-2} \nabla u  \right).$$

\medskip
\noindent
Let us denote 
$$ \displaystyle R_N(X,Y):=|X+Y|^N - |X|^N  -N|X|^{N-2}X \cdot Y \qquad (X,Y) \in \mathbb{R}^N \times \mathbb{R}^N,$$

\noindent and 
\medskip
$$
\begin{aligned} 
\displaystyle 
&    H_N(u(x),\mu_N(x)) := R_N(\nabla v_N(x), \nabla u(x)) \\
&  =  | \nabla v_N(x) + \nabla u(x) |^N - |\nabla v_N(x) |^N - N |\nabla v_N(x) |^{N-2} \nabla v_N(x) \cdot \nabla u(x).
\end{aligned}
$$

\medskip
\begin{lemma}\label{stima} For any $N\in \N$, $N\ge 2$, there exists a constant $c_N$ such that
$$
0\le R_N(X,Y) \le c_N (|Y|^N + |Y|^2 |X|^{N-2}), \quad \forall \, X,Y \in \R^N.
$$
\end{lemma}
\begin{proof}
If $N=2$, $R_2(X,Y) = |Y|^2$ does not depend on $X$ and the thesis is trivial.\\
Let $N \geq 3$ and let $X,Y \in \mathbb{R}^N.$ We consider the $C^2$ function $f:\mathbb{R} \to \mathbb{R}$ given by 
$$f(t)= \left|tY + X  \right|^N.$$

\noindent\medskip
By  a direct computation, we have
$$ \displaystyle f'(t)= N \left|tY + X  \right|^{N-2} ( tY + X) \cdot Y, $$

and
$$ \displaystyle
f''(t)=N(N-2) \left|tY + X  \right|^{N-4} \left( ( tY + X) \cdot Y \right)^2 + N \left|tY + X  \right|^{N-2}|Y|^2
$$
if $tY + X  \neq 0$, and $f''(t)=0$ otherwise.\\
Now we observe that, if $t \in [0,1]$, we have
\begin{equation}\label{fsecondo}
\displaystyle
\begin{split}
       0 \leq f''(t) 
& \leq   2c_N \left( |Y|^N + |Y|^2  |X|^{N-2} \right), \\
\end{split}
\end{equation}

\noindent where $c_N=N(N-1)2^{N-4}.$ By Taylor's theorem with Lagrange remainder, there exists $\overline{t} \in\, ]0,1[$ such that
$$ \displaystyle \frac{1}{2} f''(\overline{t})=f(1)-f(0)-f'(0)=\left|Y + X  \right|^N  - \left| X  \right|^N - N \left|X  \right|^{N-2} X \cdot Y .$$

\noindent
Taking into account \eqref{fsecondo}, we deduce that

$$ \displaystyle 0 \leq R_N(X,Y) \leq  c_N \left( |Y|^N + |Y|^2  |X|^{N-2} \right).$$
\end{proof}

\noindent
We define

\begin{align*}
\displaystyle 
W_{\mu_N}(\mathbb{R}^N):= \{ u \in L^1(\mathbb{R}^N,d\mu_N): \quad  & |\nabla u| \in L^N(\mathbb{R}^N,dx),\\                        & |\nabla u|^2 |\nabla v_N|^{N-2} \in L^1(\mathbb{R}^N,dx) \}. 
\end{align*}

\noindent
In this space we introduce the following norm:
\begin{equation*}
\lVert u \rVert_{\mu_N}:= \int_{\mathbb{R}^N} |u| \, d\mu_N + \lVert \nabla u \rVert_N + \left( \int_{\mathbb{R}^N}  |\nabla u|^2 |\nabla v_N|^{N-2} \, dx \right)^{\frac{1}{2}}  .
\end{equation*}

An immediate consequence of Lemma \ref{stima} is the following integrability condition for $H_N$.

\begin{corollary}
Let $ u \in W_{\mu_N}(\mathbb{R}^N).$ Then 
$$ \displaystyle \int_{\mathbb{R}^N} H_N(u,\mu_N) \, dx < + \infty.$$
\end{corollary}

\medskip
\noindent
It is interesting to compare $W_{\mu_N}$ with $W^{1,N}$. 

\begin{lemma}\label{W1nmun}
If $\displaystyle u \in  W^{1,N}(\mathbb{R}^N)$ is such that $|\nabla u|^2 |\nabla v_N|^{N-2} \in L^1(\mathbb{R}^N,dx)$, then $u \in W_{\mu_N}(\mathbb{R}^N).$
Moreover, we have $ W_{\mu_N}(\mathbb{R}^N) \setminus W^{1,N}(\mathbb{R}^N) \neq \emptyset$. 
\end{lemma}
\begin{proof} \text{}

\medskip
\noindent
Let $\displaystyle u \in  W^{1,N}(\mathbb{R}^N)$ such that $|\nabla u|^2 |\nabla v_N|^{N-2} \in L^1(\mathbb{R}^N,dx)$.
\begin{align*}
\displaystyle 
        \int_{\mathbb{R}^N} |u(x)| \, d \mu_N  =     \int_{\mathbb{R}^N} |u(x)|  \frac{1}{V_N \left( 1+ |x|^{\frac{N}{N-1}} \right)^N}  \, dx.
\end{align*}
By assumption $\lVert u \rVert_N < + \infty$, while
\begin{align*}
\displaystyle
    \gamma_N:=\int_{\mathbb{R}^N} \frac{1}{\left( 1+ |x|^{\frac{N}{N-1}} \right)^{\frac{N^2}{N-1}}}  \, dx < + \infty.
\end{align*}
By H\"older inequality we have
$$ \displaystyle \int_{\mathbb{R}^N} |u(x)| \, d \mu_N  \leq \frac{\gamma_N^{\frac{N-1}{N}}}{V_N} \lVert u \rVert_N  < + \infty.$$
Therefore $u \in W_{\mu_N}(\mathbb{R}^N).$\\
Finally, the constant $u \equiv 1$ is such that $u \in W_{\mu_N}(\mathbb{R}^N) \setminus W^{1,N}(\mathbb{R}^N)$.
\end{proof}

{ The next lemma shows that the condition $|\nabla u|^2 |\nabla v_N|^{N-2}\in L^1(\R^N,dx)$ is necessary for the integrability  of  $H(u,\mu_N)$,  at least in even dimension. 
\begin{lemma}\label{HBelow}
If $N\ge 4$ is even, then 
$$
H(u,\mu_N) \ge \frac{N}{2}|\nabla u|^{2}|\nabla v_N|^{N-2}.
$$
\end{lemma}
\begin{proof}
The proof is based on the following inequalities, which hold for $a,b \in \R$ and $k\in \N$, with $a\ge 0$ and $a+b\ge0$:
\begin{enumerate}
\item\label{Ineq1} if $k\ge 2$, then $(a+b)^k\ge a^{k}+ k a^{k-1} b + b^k$.  
\item\label{Ineq2}  if $k\ge 3$, then $(a+b)^k\ge a^{k}+k a^{k-1} b + k a b^{k-1} + b^k$. 
\end{enumerate}
Both $(1)$ and $(2)$ can be easily proved by induction on $k$.
Let $X,Y\in \R^N$. If $\frac{N}{2}$ is even, using inequality \eqref{Ineq1} with $k = \frac{N}{2}$, $a = |X|^2$ and $b = 2 X\cdot Y +|Y|^2$ we get that 
$$\begin{aligned}
R_N(X,Y)& = \left(|X|^2+2X\cdot Y +|Y|^2\right)^{\frac{N}{2}}-|X|^N - N |X|^{N-2} X\cdot Y \\
&  \ge \frac{N}{2} |X|^{N-2}|Y|^2 +   \left( 2 X \cdot Y + |Y|^2\right)^\frac{N}{2}. 
\end{aligned}
$$
Since $\frac{N}{2}$ is even, we get 
\begin{equation}\label{RNBelow}
R_N(X,Y)\ge \frac{N}{2}|X|^{N-2} |Y|^2. 
\end{equation}

Similarly, if $\frac{N}{2}$ is odd (hence $N\ge 6$), we can apply inequality \eqref{Ineq2} (with $k, a$ and $b$ as above) to get:
$$\begin{aligned}
R_N(X,Y)&  \ge \frac{N}{2} |X|^{N-2}|Y|^2+   \frac{N}{2}|X|^2 \left( 2 X \cdot Y + |Y|^2\right)^{\frac{N}{2}-1} +   \left( 2 X \cdot Y + |Y|^2\right)^\frac{N}{2}  \\
& =  \frac{N}{2} |X|^{N-2}|Y|^2  +  \left( 2 X \cdot Y + |Y|^2\right)^{\frac{N}{2}-1} \left( \frac{N}{2}|X|^2  + 2 X \cdot Y + |Y|^2\right)
\\
& \ge  \frac{N}{2} |X|^{N-2}|Y|^2  +  \left( 2 X \cdot Y + |Y|^2\right)^{\frac{N}{2}-1} |X+Y|^2.
\end{aligned}
$$
Since $\frac{N}{2}-1$ is even, we get again \eqref{RNBelow}. 
Hence \eqref{RNBelow} holds whenever $\frac{N}{2}$ is integer and the conclusion follows by taking $X = \nabla v_N$ and $Y = \nabla u$.  
\end{proof}
}

\begin{remark}
If $N \geq 3$ and $u \in W^{1,N}(\mathbb{R}^N)$, in general the integrability condition $|\nabla u|^2 |\nabla v_N|^{N-2} \in L^1(\mathbb{R}^N,dx)$ of Lemma \ref{W1nmun} is not satisfied. Indeed, let 
$$ \displaystyle u(x):= \sum_{k=2}^{\infty}  \frac{1}{k \sqrt{\ln k } } w(2(|x| - k)), $$
where
$$ \displaystyle
w(r):=
\begin{cases}
1+r & \text{if } r \in [-1,0[, \\
1-r & \text{if } r \in [0,1], \\
0   & \text{otherwise}.
\end{cases}
$$

{A direct computation shows that  the function $u$  satisfies  
\begin{equation}\label{example}
u \in W^{1,N}(\mathbb{R}^N) \quad \text{ and } \quad |\nabla u|^2 |\nabla v_N|^{N-2} \not\in L^1(\mathbb{R}^N,dx). 
\end{equation} 
In particular, if $N$ is even, Lemma \ref{HBelow} yields
$$
\int_{\R^N} H(u,\mu_N) \,dx = +\infty.  
$$
To prove \eqref{example}  we observe that the functions $x\rightarrow w(2(|x|-k))$ have disjoint support as $k$ varies in $\N$ with $k\ge 2$. Then 
$$ \displaystyle |u(x)|^N= \sum_{k=2}^{\infty}  \frac{1}{k^N (\ln k)^{\frac{N}{2}} }  (w(2(|x| - k)))^N.$$
Integrating in polar coordinates and using that $0\le w\le 1$, we find 
$$
\begin{aligned}
 \int_{\mathbb{R}^N} |u|^N \, dx 
& \leq \omega_{N-1} \sum_{k=2}^{\infty}  \frac{1}{k^N (\ln k)^{\frac{N}{2}} } \int_{k - \frac{1}{2}}^{k + \frac{1}{2}} \rho^{N-1} \, d \rho \\
& = \frac{\omega_{N-1}}{N} \sum_{k=2}^{\infty}  \frac{1}{k^N (\ln k)^{\frac{N}{2}} } \left[ \left(k + \frac{1}{2}   \right)^N - \left(k - \frac{1}{2}   \right)^N  \right] \\
& \leq c_1(N) \sum_{k=2}^{\infty} \frac{1}{k (\ln k)^{\frac{N}{2}} } < + \infty
\end{aligned}
$$
since $\frac{N}{2}>1$. Similarly, we have 
$$ \displaystyle |\nabla u(x)|=2 \sum_{k=2}^{\infty}  \frac{1}{k \sqrt{\ln k } }  | w'(2(|x| - k))|,$$
and 
$$ \displaystyle \int_{\mathbb{R}^N} |\nabla u(x)|^N \, dx =  c_2(N) \sum_{k=2}^{\infty}  \frac{1}{k^N (\ln k)^{\frac{N}{2}} } \int_{k - \frac{1}{2}}^{k + \frac{1}{2}} \rho^{N-1} \, d \rho < + \infty.$$

\noindent
Now, we observe that $|\nabla v_N(x)| \sim \frac{N^2}{N-1} \frac{1}{|x|}$ if $|x| >> 1$, hence, taking $k_0$ sufficiently large, we obtain
\begin{align*} \displaystyle 
                        \int_{\mathbb{R}^N} |\nabla u |^2 |\nabla v_N |^{N-2} \, dx 
& \geq  c_3(N) + c_4(N) \sum_{k=k_0}^{\infty} \frac{1}{k^2 \ln k } \int_{k - \frac{1}{2}}^{k + \frac{1}{2}} \frac{1}{\rho^{N-2}} \rho^{N-1} \, d \rho \\
& \geq                  c_5(N) \left( 1 + \sum_{k=k_0}^{\infty} \frac{1}{k \ln k } \right) =                     + \infty.
\end{align*}
}
\end{remark}

\bigskip
Now, we want to prove the density result of Theorem \ref{densita}.  Let us define  $$\mathcal{X}:= \overline{C_0^{\infty} (\mathbb{R}^N)}^{\lVert \cdot \rVert_{\mu_N}}.$$ Clearly, $ C_0^{\infty} (\mathbb{R}^N) \subset W_{\mu_N}(\mathbb{R}^N)$, so that $\mathcal{X} \subseteq W_{\mu_N}(\mathbb{R}^N).$ Our aim is to prove that the two spaces coincide. 
We will prove this in two steps. Firstly, by using a similar idea of Tintarev and Fieseler in \cite[Remark 2.2]{TF}, we show that bounded functions in $W_{\mu_N}(\mathbb{R}^N)$  belong to $\mathcal{X}.$ 

\begin{lemma}\label{funzionilimitate}

Let $u \in W_{\mu_N}(\mathbb{R}^N) \cap L^{\infty}(\mathbb{R}^N)$. Then $u \in \mathcal{X}.$
\end{lemma} 

\begin{proof}
For any $k \in \mathbb{N},$ $k \geq 2$, let us denote  $$  k^*:= \frac{1}{\left( 1 - \frac{1}{\sqrt[k]{k}} \right)^k}.$$ 
We define
\begin{equation*}
\eta_k(x):=
\begin{cases}
 \displaystyle \frac{1}{\sqrt[k]{k}} & \text{ if } |x| \leq 1, \bigskip \\
\displaystyle  \frac{1}{\sqrt[k]{|x|}} + \frac{1}{\sqrt[k]{k}} - 1 & \text{ if }  1 < |x| \leq k^*, \bigskip \\
               0 & \text{ if }  k^* < |x| .
\end{cases}
\end{equation*}

\medskip Let $u \in W_{\mu_N}(\mathbb{R}^N) \cap L^{\infty}(\mathbb{R}^N)$. We define $u_k:=u \, \eta_k.$
Observe that
%
\begin{equation*}
\nabla u_k(x):=
\begin{cases}
 \displaystyle \frac{1}{\sqrt[k]{k}} \nabla u(x) & \text{ if } |x| < 1, \bigskip \\
\displaystyle   \left( \frac{1}{\sqrt[k]{k}} + \frac{1}{\sqrt[k]{|x|}} - 1 \right)\nabla u(x) -\frac{1}{k} \frac{u(x)}{|x|^{\frac{1}{k}+1}} \frac{x}{|x|} & \text{ if }  1 < |x| < k^*, \bigskip \\
               0 & \text{ if }  k^* < |x| .
\end{cases}
\end{equation*}

\noindent
First of all, since $u \in L^{\infty}(\mathbb{R}^N)$ and $\eta_k(x) \to 1 $ for any $x \in \mathbb{R}^N$ as $k \to + \infty,$  we can apply   dominated convergence theorem to say that 
$$ \displaystyle \int_{\mathbb{R}^N} |u - u_k(x)| \, d \mu_N  \to 0 \quad \text{ as } k \to + \infty.$$
Now we want to prove 
$$ \displaystyle \lVert \nabla ( u - u_k )  \rVert_N \to 0 \quad \text{ as } k \to + \infty.$$
We will denote with $o(1)$ any quantity that goes to $0$ as $k \to + \infty.$

\begin{align*}
\displaystyle 
& \int_{\mathbb{R}^N} | \nabla u - \nabla u_k|^N \, dx \\
& = \int_{|x| < 1} | \nabla u - \nabla u_k|^N \, dx +  \int_{1 < |x| < k^*} | \nabla u - \nabla u_k|^N \, dx + \int_{ |x| > k^*} | \nabla u |^N \, dx.
\end{align*}

\noindent
We have
$$ \displaystyle \int_{|x| < 1} | \nabla u - \nabla u_k|^N \, dx= \left( 1 - \frac{1}{\sqrt[k]{k}}  \right)^N \int_{|x| < 1} | \nabla u|^N \, dx=o(1).$$

\noindent
Since $\lVert \nabla u \rVert_N < + \infty$ and $k^* \to + \infty$ as $k \to + \infty,$ we easily obtain 
$$ \displaystyle \int_{|x| > k^*} | \nabla u - \nabla u_k|^N \, dx=  \int_{|x| > k^*} | \nabla u|^N \, dx=o(1).$$
Since $u \in L^{\infty}(\mathbb{R}^N),$ there exists a constant $M_u$ such that $|u(x)| \leq M_u$ for almost every $ x \in \mathbb{R}^N$.
\begin{align*}
\displaystyle
&       \int_{1 < |x| < k^*} | \nabla u - \nabla u_k|^N \, dx  \\
& =    \int_{1 < |x| < k^*} \left| \left(  2 - \frac{1}{\sqrt[k]{k}} - \frac{1}{\sqrt[k]{|x|}}  \right) \nabla u(x)  + \frac{1}{k} \frac{u(x)}{|x|^{\frac{1}{k}+1}} \frac{x}{|x|}    \right|^N  \, dx \\
& \leq C \int_{1 < |x| < k^* }  \left(  2 - \frac{1}{\sqrt[k]{k}} - \frac{1}{\sqrt[k]{|x|}}  \right)^N  \left| \nabla u(x) \right|^N \, dx \\
& + C \frac{M_u^N}{k^N} \int_{1 < |x| < k^*} \frac{1}{|x|^{\frac{N}{k}+N}} \, dx.
\end{align*}

\noindent
Now we observe that 
$$
\displaystyle
 \frac{M_u^N}{k^N} \int_{ 1 < |x| < k^*} \frac{1}{|x|^{\frac{N}{k}+N}} \, dx < \frac{c_1}{k^N} \int_1^{+ \infty} \frac{1}{\rho^{\frac{N}{k}+1}} \, d \rho = \frac{c_2}{k^{N-1}}=o(1).
$$

\noindent
Moreover, we can apply dominated convergence theorem to say that
$$ \displaystyle \int_{1 < |x| < k^* }  \left(  2 - \frac{1}{\sqrt[k]{k}} - \frac{1}{\sqrt[k]{|x|}}  \right)^N  \left| \nabla u(x) \right|^N \, dx \to 0 \quad \text{ as } k \to + \infty.$$

\noindent
Hence
$$ \displaystyle \lVert \nabla ( u - u_k )  \rVert_N \to 0 \quad \text{ as } k \to + \infty.$$

\noindent
It remains to prove, in the case $N \geq 3$, that
$$ \displaystyle \left( \int_{\mathbb{R}^N}  |\nabla u-  \nabla u_k |^2 |\nabla v_N|^{N-2} \, dx \right)^{\frac{1}{2}} \to 0 \quad \text{ as } k \to + \infty.$$

\noindent
We observe that $\displaystyle | \nabla v_N(x)| \sim \frac{1}{|x|}$ as $|x| \to + \infty,$ by which there exists a constant $ r_N > 1$ such that $ \displaystyle  |\nabla v_N (x)| \leq \frac{2}{|x|} $ if $|x| \geq r_N.$ Since $k \to + \infty$ implies $k^* \to + \infty,$ we can assume from the beginning that $k$ is such that $k^* > r_N$.
Moreover, we recall that $|\nabla v_N|$ is bounded from above.

\begin{align*}
\displaystyle
&       \int_{\mathbb{R}^N}  |\nabla u - \nabla u_k|^2 |\nabla v_N|^{N-2} \, dx \\
& =     \int_{ |x| < 1 }  |\nabla u - \nabla u_k|^2 |\nabla v_N|^{N-2} \, dx +  \int_{ 1 < |x| < r_N }  |\nabla u - \nabla u_k|^2 |\nabla v_N|^{N-2} \, dx \\
& +    \int_{ r_N < |x|  < k^* }  |\nabla u - \nabla u_k|^2 |\nabla v_N|^{N-2} \, dx + \int_{ k^* < |x|  }  |\nabla u|^2 |\nabla v_N|^{N-2} \, dx.
\end{align*}

\noindent
Reasoning in a similar way to the previous case, we obtain 
$$ \displaystyle \int_{ |x| < 1 }  |\nabla u - \nabla u_k|^2 |\nabla v_N|^{N-2} \, dx=o(1),$$
$$ \displaystyle   \int_{1 < |x| < r_N} | \nabla u - \nabla u_k|^2 |\nabla v_N|^{N-2} \, dx  =o(1),$$
and
$$ \displaystyle \int_{ k^* < |x|  }  |\nabla u |^2 |\nabla v_N|^{N-2} \, dx=o(1).$$
Now we observe that
\begin{align*}
\displaystyle
&       \int_{r_N < |x| < k^*} | \nabla u - \nabla u_k|^2 |\nabla v_N|^{N-2} \, dx  \\
& =    \int_{r_N < |x| < k^*} \left| \left(  2 - \frac{1}{\sqrt[k]{k}} - \frac{1}{\sqrt[k]{|x|}}  \right) \nabla u(x)  +\frac{1}{k} \frac{u(x)}{|x|^{\frac{1}{k}+1}} \frac{x}{|x|}    \right|^2 |\nabla v_N|^{N-2} \, dx \\
& \leq C \int_{r_N < |x| < k^*}  \left(  2 - \frac{1}{\sqrt[k]{k}} - \frac{1}{\sqrt[k]{|x|}}  \right)^2  \left| \nabla u(x) \right|^2 |\nabla v_N|^{N-2}  \, dx \\
& + C \frac{1}{k^2} \int_{r_N < |x|} \frac{1}{|x|^{\frac{2}{k}+N}} \, dx . \\
\end{align*}

\noindent
Clearly, 
$$ \displaystyle \frac{1}{k^2} \int_{r_N < |x|} \frac{1}{|x|^{\frac{2}{k}+N}} \, dx=o(1).$$

\noindent
Moreover, taking into account that $ \displaystyle |\nabla u|^2 |\nabla v_N|^{N-2} \in L^1(\mathbb{R}^N,dx),$  we can apply dominated convergence theorem to say that 
$$ \displaystyle  \int_{r_N < |x| < k^*}  \left(  2 - \frac{1}{\sqrt[k]{k}} - \frac{1}{\sqrt[k]{|x|}}  \right)^2  \left| \nabla u(x) \right|^2 |\nabla v_N|^{N-2}  \, dx=o(1).$$

\noindent
Finally, we have 
$$ \displaystyle \left( \int_{\mathbb{R}^N}  |\nabla u-  \nabla u_k |^2 |\nabla v_N|^{N-2} \, dx \right)^{\frac{1}{2}} \to 0 \quad \text{ as } k \to + \infty.$$

\noindent
We have proved that $\{u_k \}_k $ is a sequence of compactly supported functions such that $\lVert u -u_k \rVert_{\mu_N} \to 0$ as $k \to +\infty$.\\
We want to take a sequence $\{\tilde{u}_k \}_k \subset C_0^{\infty}(\mathbb{R}^N)$ such that $\lVert u -\tilde{u}_k \rVert_{\mu_N} \to 0$ as $k \to +\infty$.\\
Let us denote with $B_{k^*}$ the open ball $B_{k^*}(O)$.\\
Since ${u_{k}}_{|_{B_{k^*}}} \in W_0^{1,N}(B_{k^*}),$ there exists $\tilde{u}_k \in C_0^{\infty}(B_{k^*})$ such that 
$$ \displaystyle \left\lVert {u_{k}}_{|_{B_{k^*}}}- \tilde{u}_k  \right\rVert_{L^N(B_{k^*})} + \left\lVert \nabla \left({u_{k}}_{|_{B_{k^*}}}- \tilde{u}_k \right) \right\rVert_{L^N(B_{k^*})} < \frac{1}{k}.$$
We can extend to $0$ in $\mathbb{R}^N \setminus B_{k^*}$  and consider $\tilde{u}_k \in C_0^{\infty}(\mathbb{R}^N)$. Hence
$$ \displaystyle \lVert u_k - \tilde{u}_k  \rVert_N + \lVert \nabla \left(u_k - \tilde{u}_k \right) \rVert_N  < \frac{1}{k}.$$

\noindent
Since
\begin{align*}
\displaystyle
    \gamma_N:=\int_{\mathbb{R}^N} \frac{1}{\left( 1+ |x|^{\frac{N}{N-1}} \right)^{\frac{N^2}{N-1}}}  \, dx < + \infty  ,
\end{align*}

\noindent
we can apply H\"older inequality and say that
\begin{align*}
 \displaystyle 
       \int_{\mathbb{R}^N} |u_k(x) - \tilde{u}_k(x)| \, d \mu_N  < \gamma_N^{\frac{N-1}{N}} \lVert u_{k}- \tilde{u}_k  \rVert_N < \frac{\gamma_N^{\frac{N-1}{N}}}{k}=o(1).
\end{align*}

\noindent
Now let $k$ such that $k^* > r_N.$\\
Since
\begin{align*}
\displaystyle 
    \int_{0 < |x| \leq k^*} |\nabla v_N|^N \, dx \leq C \left( \ln k^* + 1   \right),
\end{align*}

we have
\begin{align*}
\displaystyle  \int_{\mathbb{R}^N}  |\nabla(u_k - \tilde{u}_k)|^2 |\nabla v_N|^{N-2} \, dx  <    C \frac{ \left( \ln k^* + 1    \right)^{\frac{N-2}{N}} }{k^2}
\end{align*}

\noindent
Now we observe that, as $k \to + \infty$
$$ \displaystyle \left( \ln k^* + 1 \right)^{\frac{N-2}{N}} \sim k^{\frac{N-2}{N}} \left(- \ln \left(  1 - \frac{1}{\sqrt[k]{k}} \right)  \right)^{\frac{N-2}{N}}, $$

by which
$$ \displaystyle \frac{  \left( \ln k^* + 1   \right)   ^{\frac{N-2}{N}} }{k^2} \sim \frac{\left(- \ln \left(  1 - \frac{1}{\sqrt[k]{k}} \right)  \right)^{\frac{N-2}{N}}}{k^{1 + \frac{2}{N}}} \to 0.$$

Hence
$$ \displaystyle \int_{\mathbb{R}^N}  |\nabla(u_k - \tilde{u}_k)|^2 |\nabla v_N|^{N-2} \, dx \to 0.$$

\noindent
Finally, we have

\begin{align*}
\displaystyle
  \lVert u - \tilde{u}_k \rVert_{\mu_N} \leq  \lVert u -  u_k \rVert_{\mu_N} +  \lVert u_k - \tilde{u}_k \rVert_{\mu_N} \to 0 \quad \text{as } k \to + \infty.  
\end{align*}

\end{proof}

\noindent
Thanks to previous lemma, it remains to prove that any function $u \in W_{\mu_N}(\mathbb{R}^N)$ can be approximated by a sequence of bounded functions in $ W_{\mu_N}(\mathbb{R}^N)$.

\begin{lemma}\label{convergenzatroncate}\

\medskip
\noindent
Let $u \in W_{\mu_N}(\mathbb{R}^N)$. There exists a sequence $\{u_k\}_{k \in \mathbb{N}} \subset W_{\mu_N}(\mathbb{R}^N) \cap L^{\infty}(\mathbb{R}^N) $ such that $\lVert u-u_{k} \rVert_{\mu_N} \to 0$ as $k \to + \infty$. 
\end{lemma}

\begin{proof}

Let $ u \in W_{\mu_N}(\mathbb{R}^N).$ Taking into account \cite[Lemma A.4, Chapter 2]{KS}, we have $|u| \in  W_{\mu_N}(\mathbb{R}^N)$ with
\begin{equation*}
\nabla |u|=
\begin{cases}
 \nabla u  & \text{ a.e. in } \{ x \in \mathbb{R}^N  \; | \; u(x)>0\},\\
 0         & \text{ a.e. in } \{ x \in \mathbb{R}^N  \; | \; u(x)=0\},\\
 -\nabla u & \text{ a.e. in } \{ x \in \mathbb{R}^N  \; | \; u(x)<0\}.
\end{cases}
\end{equation*}

\noindent
For any $\lambda > 0$, we can define the \textit{truncated function of $u$ of parameter $\lambda$} as $\displaystyle u_{\lambda}:=\max\{-\lambda,\min\{u,\lambda\}\},$ that is

\begin{equation*}
u_{\lambda}(x)=
\begin{cases}
   \lambda         & \text{ a.e. in } \{ x \in \mathbb{R}^N  \; | \; u(x)\geq  \lambda \},\\
    u              & \text{ a.e. in } \{ x \in \mathbb{R}^N  \; | \; - \lambda < u(x) < \lambda \},\\
 - \lambda         & \text{ a.e. in } \{ x \in \mathbb{R}^N  \; | \; u(x) \leq - \lambda \}.
\end{cases}
\end{equation*}

\noindent
Let $k \in \mathbb{N},$ $k \geq 1,$ and let $\{u_k\}_{k \in \mathbb{N}}$ a sequence of truncated function of $u$ of parameter $k$. Clearly,  $\{u_k\}_{k \in \mathbb{N}} \subset W_{\mu_N}(\mathbb{R}^N) \cap L^{\infty}(\mathbb{R}^N) $.\\
Let us denote with $\displaystyle E_k:=\{ x \in \mathbb{R}^N \; | \; |u(x)| \geq k \}.$\\
Hence 
$$ \displaystyle \int_{\mathbb{R}^N} |u - u_k| \, d \mu_N = \int_{E_k} |u - u_k | \, d \mu_N \leq \int_{E_k} |u| \, d \mu_N. $$ 

\noindent
Since $u \in L^1(\mathbb{R}^N, d \mu_N),$ by Chebyshev's inequality, we know that 
$$|E_k|_{\mu_N} \to 0 \qquad \text{as } k \to + \infty.$$
Moreover, we can apply absolute continuity of the integral and say that 
$$ \displaystyle \int_{E_k} |u| \, d \mu_N \to 0 \qquad \text{ as } k \to + \infty.$$

\noindent
Therefore, we deduce $u_k \to u$ in $L^1(\mathbb{R}^N, d\mu_N)$.

\medskip
\noindent
Let now $\varepsilon > 0$ and let us consider the set

$$ \displaystyle \left\{ x \in \mathbb{R}^N \; : \; \left| \frac{ \partial u_k}{ \partial x_i}(x) - \frac{ \partial u}{ \partial x_i}(x)   \right| \geq \varepsilon  \right\}. $$ 

\noindent
Since

$$ \displaystyle \left\{ x \in \mathbb{R}^N \; : \; \left| \frac{ \partial u_k}{ \partial x_i}(x) - \frac{ \partial u}{ \partial x_i}(x)   \right| \geq \varepsilon  \right\} \subset E_k, $$ 

\noindent
we deduce that 

$$ \displaystyle \frac{ \partial u_k}{ \partial x_i} \to \frac{ \partial u}{ \partial x_i} \quad \text{ in measure with respect to $\mu_N$.}$$

\noindent
Hence there exists a subsequence $\{u_{k_j}\}_{j \in \mathbb{N}}$ such that $ \nabla u_{k_j} \to \nabla u $ almost everywhere in $\mathbb{R}^N$ with respect to the measure $\mu_N$. Since the Lebesgue measure $dx$ is absolutely continuous with respect to the measure $d\mu_N$, we deduce that
 $ \nabla u_{k_j} \to \nabla u $ almost everywhere in $\mathbb{R}^N$ also with respect to the Lebesgue measure.

\noindent
Finally, since 

$$ \displaystyle \left| \nabla u(x) - \nabla u_{k_j}(x) \right|^N \leq \left| \nabla u(x) \right|^N \quad \text{a.e. } x \in \mathbb{R}^N,$$

and

$$ \displaystyle  |\nabla u(x)-  \nabla u_{k_j}(x) |^2 |\nabla v_N(x)|^{N-2} \leq |\nabla u(x) |^2 |\nabla v_N(x)|^{N-2} \quad \text{a.e. } x \in \mathbb{R}^N,$$

\noindent
we can apply dominated convergence theorem to say that

$$ \lVert \nabla u - \nabla u_{k_j} \rVert_N \to 0 \quad \text{ as } j \to + \infty,$$

and

$$ \displaystyle \left( \int_{\mathbb{R}^N}  |\nabla u-  \nabla u_{k_j} |^2 |\nabla v_N|^{N-2} \, dx \right)^{\frac{1}{2}} \to 0 \quad \text{ as } j \to + \infty.$$

\end{proof}

\noindent
Taking into account Lemma \ref{funzionilimitate} and Lemma \ref{convergenzatroncate}, we have proved Theorem \ref{densita}, that is 
$$ \displaystyle W_{\mu_N}(\mathbb{R}^N)=\mathcal{X}=\overline{C_0^{\infty} (\mathbb{R}^N)}^{\lVert \cdot \rVert_{\mu_N}}.$$

\medskip
Now, we can extend the $N$-dimensional Euclidean Onofri inequality to the weighted Sobolev space $W_{\mu_N}(\mathbb{R}^N)$.

\begin{proof}[Proof of Theorem \ref{EuclideanOnofriNdim}]

\noindent
Let $u \in W_{\mu_N}(\mathbb{R}^N)$. Since $W_{\mu_N}(\mathbb{R}^N)= \overline{C_0^{\infty} (\mathbb{R}^N)}^{\lVert \cdot \rVert_{\mu_N}},$ there exists a sequence $\displaystyle \{u_m\}_{m \in \mathbb{N}} \subset C_0^{\infty}(\mathbb{R}^N)$ such that $\displaystyle \lVert u-u_m \rVert_{\mu_N} \to 0.$ By \eqref{EuclideanOnofriNdimensional}, we have
$$ \displaystyle \ln \left( \int_{\mathbb{R}^N} e^{u_m} \, d\mu_N \right) \leq  \frac{1}{\widetilde{\omega_N}} \int_{\mathbb{R}^N} H_N(u_m,\mu_N) \, dx + \int_{\mathbb{R}^N} u_m \, d\mu_N $$
for any  $ m \in \mathbb{N}$.\\
The condition $\displaystyle \lVert u-u_m \rVert_{\mu_N} \to 0$ implies that there exists a subsequence, that we recall $\{u_m\}_{m \in \mathbb{N}}$, and a nonnegative function 
$w \in L^N(\mathbb{R}^N,dx)$ with $w^2 |\nabla v_N|^{N-2} \in L^1(\mathbb{R}^N,dx)$, such that 
\begin{itemize}
\item[•] $u_m(x) \to u(x)$ a.e. in $\mathbb{R}^N$ as $m \to + \infty$;
\item[•] $\nabla u_m(x) \to \nabla u(x)$ a.e. in $\mathbb{R}^N$ as $m \to + \infty$;
\item[•] for any m, $|\nabla u_m(x) | \leq w(x)$ a.e. in $\mathbb{R}^N$. 
\end{itemize}
We easily obtain that
$$ \displaystyle \int_{\mathbb{R}^N} u_m \, d\mu_N  \to \int_{\mathbb{R}^N} u \, d\mu_N \quad \text{as } m \to + \infty.$$
It's easy to observe that
$$ \displaystyle H_N(u_m(x),\mu_N(x)) \to H_N(u(x),\mu_N(x)) \quad \text{ a.e. in } \mathbb{R}^N  \quad \text{ as   } \; m \to + \infty.$$
Taking into account Lemma \ref{stima}, we have
\begin{align*}
\displaystyle
  0 \leq H_N(u_m(x),\mu_N(x)) 
& \leq \frac{c_N}{2} \left( |\nabla u_m(x) |^N + |\nabla u_m(x) |^2|\nabla v_N(x) |^{N-2}   \right) \\
& \leq \frac{c_N}{2} \left( w(x)^N + w(x)^2|\nabla v_N(x) |^{N-2}   \right), 
\end{align*}
almost everywhere in $\mathbb{R}^N$ and for any $m$.\\
Since $w^N + w^2|\nabla v_N |^{N-2} \in L^1(\mathbb{R}^N,dx)$, we can apply dominated convergence theorem to say that
$$ \displaystyle \int_{\mathbb{R}^N} H_N(u_m,\mu_N) \, dx  \to \int_{\mathbb{R}^N} H_N(u,\mu_N) \, dx .$$
Finally, by Fatou's Lemma, we have 
\begin{equation*}
\displaystyle
\int_{\mathbb{R}^N} e^u \, d\mu_N 
\leq \liminf_{m \to + \infty} \left( \int_{\mathbb{R}^N} e^{u_m} \, d\mu_N \right) 
\leq  e^{\frac{1}{\widetilde{\omega_N}} \int_{\mathbb{R}^N} H_N(u,\mu_N) \, dx + \int_{\mathbb{R}^N} u \, d\mu_N},
\end{equation*}
and passing to logarithm we can conclude that 
$$ \displaystyle \ln \left( \int_{\mathbb{R}^N} e^u \, d\mu_N \right) \leq  \frac{1}{\widetilde{\omega_N}} \int_{\mathbb{R}^N} H_N(u,\mu_N) \, dx + \int_{\mathbb{R}^N} u \, d\mu_N. $$

\end{proof}

\section{Equivalence between Logarithmic Moser-Trudinger inequality and Euclidean Onofri inequality}  

\bigskip
This section is devoted to the proof Theorem \ref{equivalence}, that is to the equivalence between the sharp inequalities  \eqref{OnofriEuclideaEstesa} and \eqref{CarlesonChang}. 

\begin{proof}[Proof of Theorem \ref{equivalence}]
Let $I$ and $J$ be the functionals defined in \eqref{Ifunctional} and \eqref{Jfunctional}.  We need to prove that 
$$
\inf_{W_{\mu_N}(\R^N)} I = \inf_{W^{1,N}_0(B_1)}  J +\sum_{k=1}^{N-1}\frac{1}{k}.
$$
\noindent
We begin with the proof of $ \lq\lq \geq "$.\\
Let $u \in W_{\mu_N}(\mathbb{R}^N)$. Since $W_{\mu_N}(\mathbb{R}^N)= \overline{C_0^{\infty} (\mathbb{R}^N)}^{\lVert \cdot \rVert_{\mu_N}}$, we can assume $u \in  C_0^{\infty}(\mathbb{R}^N)$ and the thesis will follow by density.\\
Let us consider the following cutoff function $\Psi \in C^{\infty}_0(\mathbb{R}^N)$ such that $0 \leq \Psi \leq 1$, $ \Psi \equiv 1$ if $ \displaystyle |x| \leq \frac{1}{2},$ and $\Psi \equiv 0 $ if $\displaystyle  |x| \geq 1$.\\
Let $r>0$ and let us define $\displaystyle \Psi_r(x):=\Psi\left(\frac{x}{r}\right)$, so that $ \Psi_r \equiv 1$ if $ \displaystyle |x| \leq \frac{r}{2},$ and $\Psi_r \equiv 0 $ if $\displaystyle  |x| \geq r.$ Obviously, $\displaystyle \Psi_r(x) \to 1$ as $r \to +\infty,$ for any $ x \in \mathbb{R}^N.$\\
For simplicity of notation, let's say $B_r:=B_r(O);$ moreover, we will denote with $o(1)$ any function that goes to zero as $r \to + \infty.$\\
Using the characteristic function $\chi$, we can write
$$ \displaystyle \int_{B_r} e^{u(x)\Psi_r(x) + v_N(x)}   \, dx =\int_{\mathbb{R}^N} e^{u(x)\Psi_r(x)+ v_N(x)}\chi_{B_r}(x)  \, dx.$$
Now we observe that $|u|$ is bounded by a constant $M_{u}$, so for any $ x \in \mathbb{R}^N$ and for any $r>0,$
$$ \displaystyle \left| e^{u(x)\Psi_r(x) + v_N(x)}\chi_{B_r}(x) \right| \leq e^{M_{u}+v_N(x)} =e^{M_{u}} \mu_N(x) \in L^1(\mathbb{R}^N).$$
Moreover, for any $x \in \mathbb{R}^N$,
$$ \displaystyle  e^{u(x)\Psi_r(x) + v_N(x)}\chi_{B_r}(x) \to e^{u(x) + v_N(x)} \quad  \text{ as } r \to + \infty. $$
Therefore we can apply dominated convergence theorem and say that, as $r \to + \infty$,
$$ \displaystyle  \int_{B_r} e^{u(x)\Psi_r(x) + v_N(x)}   \, dx  \to  \int_{\mathbb{R}^N} e^{u(x)+v_N(x)}\, dx  = \int_{\mathbb{R}^N} e^u \, d\mu_N. $$
We can write
\begin{equation*}
\displaystyle \ln \left( \int_{\mathbb{R}^N} e^u \, d\mu_N \right)= \lim_{r \to + \infty} \ln \left(  \int_{B_r} e^{u(x)\Psi_r(x) + v_N(x)}   \, dx \right),
\end{equation*}
therefore we can estimate
$$\displaystyle  \ln \left(  \int_{B_r} e^{u(x)\Psi_r(x) + v_N(x)}   \, dx \right).$$
In general the function $u\Psi_r + v_N \notin W^{1,N}_0(B_r)$.\\
If we define $w:= u\Psi_r + v_N - v_N(r)$, (where with $v_N(r)$ we make a slight abuse of notation to denote a constant function on $B_r$ whose value is $v_N(x)$ as $|x|=r$), we have $w \in  W^{1,N}_0(B_r)$.
\begin{align*}
\displaystyle 
     \ln \left( \int_{B_r} e^{u(x)\Psi_r(x)} \, d \mu_N   \right)
& =  \ln \left( \int_{B_r} e^{u(x)\Psi_r(x) + v_N(x)}   \, dx   \right) \\
& = N\ln \left( \frac{r}{  1+ r^{\frac{N}{N-1}}  }\right) + \ln \left( \frac{1}{|B_r|} \int_{B_r} e^{w(x)}   \, dx \right) \\
& = -\frac{N}{N-1} \ln r + \ln \left( \frac{1}{|B_r|} \int_{B_r} e^{w(x)}   \, dx \right) + o(1).
\end{align*}
Since $w \in  W^{1,N}_0(B_r)$, we have 
$$ \displaystyle \inf_{W^{1,N}_0(B_1)}  J <  \frac{1}{\widetilde{\omega_N}} \int_{B_r} | \nabla w |^N \, dx - \ln \left( \frac{1}{|B_r|} \int_{B_r} e^{w(x)}   \, dx \right). $$
Therefore
\begin{align*}
\displaystyle 
\ln \left( \int_{B_r} e^{u\Psi_r} \, d \mu_N   \right) < \frac{1}{\widetilde{\omega_N}} \int_{B_r} | \nabla w |^N \, dx  -\frac{N}{N-1} \ln r - \inf_{W^{1,N}_0(B_1)}  J  + o(1). 
\end{align*}
We need to compute $ \displaystyle \frac{1}{\widetilde{\omega_N}} \int_{B_r} | \nabla w |^N \, dx$.\\
Since
$$ \displaystyle \nabla w = \nabla(u \Psi_r) + \nabla v_N,$$
and
\begin{align*}
 \displaystyle
H_N(u \Psi_r, \mu_N) = | \nabla w|^N - |\nabla v_N |^N - N |\nabla v_N |^{N-2} \nabla  v_N  \cdot \nabla(u \Psi_r),
\end{align*}
we have
$$ \displaystyle | \nabla w|^N= H_N(u \Psi_r, \mu_N)  + |\nabla v_N |^N + N |\nabla v_N |^{N-2} \nabla  v_N  \cdot \nabla(u \Psi_r),$$
by which
\begin{align*}
\displaystyle
     \int_{B_r} | \nabla w |^N \, dx 
& =  \int_{B_r} H_N(u \Psi_r, \mu_N) \, dx \\
& +  \int_{B_r} |\nabla v_N |^N \, dx + \int_{B_r} |\nabla v_N |^{N-2} \nabla  v_N  \cdot \nabla(u \Psi_r) \, dx.
\end{align*}

\noindent
By Green's formula and taking into account \eqref{GaussGreen}, we obtain
\begin{align*}
\displaystyle
\frac{N}{\widetilde{\omega_N}} \int_{B_r} |\nabla v_N |^{N-2} \nabla  v_N  \cdot \nabla(u \Psi_r) \, dx 
= \int_{B_r} u \Psi_r  \mu_N \, dx .
\end{align*}

\noindent
Now, we need to compute $\displaystyle \frac{1}{\widetilde{\omega_N}} \int_{B_r} |\nabla v_N |^N \, dx.$
\begin{align}\label{contovn}
\displaystyle
 \frac{1}{\widetilde{\omega_N}} \int_{B_r} |\nabla v_N |^N \, dx 
& = \frac{N^{2 N}}{(N-1)^N \widetilde{\omega_N}}\int_{B_r} \frac{|x|^{\frac{N}{N-1}}}{\left(1+|x|^{\frac{N}{N-1}}\right)^N } \, dx \nonumber \\
& = \int_0^{r^{\frac{N}{N-1}}} \frac{t^{N-1}}{\left( 1+t \right)^N}  \, dt \nonumber \\
& = \int_0^{r^{\frac{N}{N-1}}} \frac{1}{ 1+t}  \, dt - \sum_{k=0}^{N-2} \binom{N-1}{k} \int_0^{r^{\frac{N}{N-1}}} \frac{t^k}{\left( 1+t \right)^N}  \, dt \\
& =  \frac{N}{N-1} \ln r - \sum_{k=0}^{N-2} \binom{N-1}{k} \int_0^{r^{\frac{N}{N-1}}} \frac{t^k}{\left( 1+t \right)^N}  \, dt + o(1). \nonumber
\end{align}

\noindent
By induction, we prove that
\begin{equation}\label{ind}
\displaystyle  \binom{N-1}{k} \int_0^{r^{\frac{N}{N-1}}} \frac{t^k}{\left( 1+t \right)^N}  \, dt= \frac{1}{N-k-1} + o(1)
\end{equation}
for any  $k=0,...,N-2.$

\noindent
If $k=0,$ we have
\begin{align*}
\displaystyle
\binom{N-1}{k} \int_0^{r^{\frac{N}{N-1}}} \frac{t^k}{\left( 1+t \right)^N}  \, dt
 = \int_0^{r^{\frac{N}{N-1}}} \frac{1}{\left( 1+t \right)^N}  \, dt 
 = \frac{1}{N-1} + o(1).
\end{align*}

\noindent
Now, let $N \geq 3,$ $k \geq 1,$  and let assume $\eqref{ind}$ to hold for $k-1.$ 
\begin{align*}
\displaystyle
&   \binom{N-1}{k} \int_0^{r^{\frac{N}{N-1}}} \frac{t^k}{\left( 1+t \right)^N}  \, dt \\
& = \binom{N-1}{k} \left[ o(1) + \frac{k}{N-1} \int_0^{r^{\frac{N}{N-1}}} \frac{t^{k-1}(1+t)}{\left( 1+t \right)^N}  \, dt       \right] \\
& = \binom{N-1}{k} \left[ o(1) + \frac{\frac{k}{N-1}}{\binom{N-1}{k-1}} \left( \frac{1}{N-k} + o(1)  \right)   + \frac{k}{N-1}\int_0^{r^{\frac{N}{N-1}}} \frac{t^k}{\left( 1+t \right)^N}  \, dt      \right] \\
& = \frac{1}{N-1}  +  \frac{k}{N-1} \binom{N-1}{k} \int_0^{r^{\frac{N}{N-1}}} \frac{t^k}{\left( 1+t \right)^N}  \, dt + o(1).
\end{align*}

\noindent
Therefore
$$ \displaystyle  \binom{N-1}{k} \int_0^{r^{\frac{N}{N-1}}} \frac{t^k}{\left( 1+t \right)^N}  \, dt= \frac{1}{N-k-1} + o(1). $$

\medskip
\noindent
By $\eqref{ind}$ we deduce that
\begin{equation}\label{sum}
\displaystyle \sum_{k=0}^{N-2} \binom{N-1}{k} \int_0^{r^{\frac{N}{N-1}}} \frac{t^k}{\left( 1+t \right)^N}  \, dt=\sum_{k=1}^{N-1} \frac{1}{k} + o(1) .
\end{equation}

\noindent
Finally,
\begin{equation}\label{contoGradvnallaN}
\displaystyle \frac{1}{\widetilde{\omega_N}} \int_{B_r} |\nabla v_N |^N \, dx =\frac{N}{N-1} \ln r - \sum_{k=1}^{N-1} \frac{1}{k} + o(1). 
\end{equation}

\noindent
Summing up, we have shown that
\begin{align*}
\displaystyle 
 \inf_{W^{1,N}_0(B_1)}  J + \sum_{k=1}^{N-1} \frac{1}{k} <  I(u \Psi_r) + o(1).
\end{align*}

\noindent
Now, we pass to the limit as $r \to + \infty.$ 

\noindent
We observe that $|u|$ is bounded by a constant $M_{u}$, so for any $ x \in \mathbb{R}^N$ and for any $r>0,$
$$ \displaystyle \left|  u(x) \Psi_r(x)  \mu_N(x) \right| \leq M_{u} \mu_N(x) \in L^1(\mathbb{R}^N).$$
Moreover, for any $x \in \mathbb{R}^N$,
$$ \displaystyle   u(x) \Psi_r(x)  \mu_N(x) \to u(x)  \mu_N(x) \quad  \text{ as } r \to + \infty. $$
We can apply dominated convergence theorem to say that
$$ \displaystyle   \int_{\mathbb{R}^N} u \Psi_r  \, d\mu_N  \to  \int_{\mathbb{R}^N} u  d\mu_N. $$

\noindent
Now we observe that, for any $x \in \mathbb{R}^N$,
$$ \displaystyle H_N(u(x) \Psi_r(x), \mu_N(x))  \to H_N(u(x), \mu_N(x)) \qquad \text{ as } r \to + \infty.$$

\noindent
Moreover, taking into account Lemma \ref{stima} and $u \in  C_0^{\infty}(\mathbb{R}^N)$, there exists a function $v \in  L^1(\mathbb{R}^N, dx)$ such that, for any $ x \in \mathbb{R}^N$ and for any $r>1$, we have 
\begin{align*}
\displaystyle
        0 \leq  H_N(u(x) \Psi_r(x), \mu_N(x)) \leq v(x).
\end{align*}

\noindent
Therefore we can apply dominated convergence theorem and say that
$$ \displaystyle \frac{1}{\widetilde{\omega_N}} \int_{\mathbb{R}^N} H_N(u \Psi_r, \mu_N) \, dx  \to \frac{1}{\widetilde{\omega_N}} \int_{\mathbb{R}^N} H_N(u, \mu_N) \, dx .$$

\noindent
Finally,
\begin{align*}
\displaystyle 
 \inf_{W^{1,N}_0(B_1)}  J + \sum_{k=1}^{N-1} \frac{1}{k} \leq  I(u),
\end{align*}

\noindent
by which
$$
 \inf_{W^{1,N}_0(B_1)}  J +\sum_{k=1}^{N-1}\frac{1}{k} \leq \inf_{C_0^{\infty}(\R^N)} I.
$$

\noindent
By density, that is by the proof of Theorem \ref{EuclideanOnofriNdim}, we obtain
$$
 \inf_{W^{1,N}_0(B_1)}  J +\sum_{k=1}^{N-1}\frac{1}{k} \leq \inf_{W_{\mu_N}(\R^N)} I.
$$

\noindent
Now, we prove $ \lq\lq\leq "$\\
Let $u  \in W^{1,N}_0(B_1)$ and let $r>0$. We define the function $u_r: \mathbb{R}^N \to \mathbb{R}$ as 
$$ \displaystyle 
u_r(x)=
\begin{cases}
\displaystyle u \left( \frac{x}{r} \right) - v_N(x)+v_N(r) & \text{if $|x| \leq r$,} \bigskip \\
0 & \text{if $|x| > r$.}
\end{cases}
$$
By Lemma \ref{W1nmun}, we have $u_r \in W_{\mu_N}(\mathbb{R}^N)$, by which
$$ \displaystyle  \inf_{W_{\mu_N}(\R^N)} I \leq I(u_r),
$$
where
\begin{equation}\label{calcolo}
\displaystyle I(u_r) = \frac{1}{\widetilde{\omega_N}} \int_{\mathbb{R}^N} H_N(u_r,\mu_N) \, dx + \int_{\mathbb{R}^N} u_r \, d\mu_N - \ln \left( \int_{\mathbb{R}^N} e^{u_r} \, d\mu_N \right).
\end{equation}
We need to compute the terms in \eqref{calcolo}.
$$
\displaystyle 
\ln \left( \int_{\mathbb{R}^N} e^{u_r} \, d\mu_N \right) = \ln \left( \mu_N(r) \int_{B_r} e^{u \left( \frac{x}{r} \right)} \, dx + \int_{\mathbb{R}^N \setminus B_r}  \mu_N(x) \, dx \right);
$$

\begin{align*}
\displaystyle
   \mu_N(r) \int_{B_r} e^{u \left( \frac{x}{r} \right)} \, dx  = \left( \frac{r}{1+r^{\frac{N}{N-1}}}  \right)^N \frac{1}{V_N} \int_{B_1} e^{u(y)}  \, dy;
\end{align*}

\begin{align*}
\displaystyle
    \int_{\mathbb{R}^N \setminus B_r}  \mu_N(x) \, dx = \left[ \left( \frac{t}{1+t} \right)^{N-1} \right]_{r^{\frac{N}{N-1}}}^{+ \infty}  = 1 - \frac{r^N}{\left( 1+r^{\frac{N}{N-1}}  \right)^{N-1}}.
\end{align*}

\noindent
Therefore
\begin{align*}
\displaystyle
&  \ln \left( \int_{\mathbb{R}^N} e^{u_r} \, d\mu_N \right) \\
& = \ln \left( \left( \frac{r}{1+r^{\frac{N}{N-1}}}  \right)^N \frac{1}{V_N} \int_{B_1} e^{u}  \, dx + 1 - \frac{r^N}{\left( 1+r^{\frac{N}{N-1}}  \right)^{N-1}} \right) \\
& = -\frac{N}{N-1} \ln r + \ln \left( \frac{1}{V_N} \int_{B_1} e^{u}  \, dx + N-1 + o(1) \right) + o(1).\\
\end{align*}

\noindent
Now, we compute $ \displaystyle \frac{1}{\widetilde{\omega_N}} \int_{\mathbb{R}^N} H_N(u_r,\mu_N) \, dx.$\\
We have
\begin{align*}
\displaystyle 
    \frac{1}{\widetilde{\omega_N}} \int_{\mathbb{R}^N} H_N(u_r,\mu_N) \, dx 
& = \frac{1}{\widetilde{\omega_N}} \int_{B_1} |\nabla u |^N \, dx + \frac{N-1}{\widetilde{\omega_N}} \int_{B_r} |\nabla v_N|^N \, dx \\
& - \frac{N}{ \widetilde{\omega_N}} \int_{B_r} |\nabla  v_N |^{N-2} \nabla  v_N \cdot \nabla_x u \left( \frac{x}{r} \right) \, dx. \\
\end{align*}

\noindent
Taking into account \eqref{contoGradvnallaN}, we have 
$$ \displaystyle \frac{N-1}{\widetilde{\omega_N}} \int_{B_r} |\nabla v_N|^N \, dx= N \ln r - (N-1)\sum_{k=1}^{N-1} \frac{1}{k}  + o(1). $$

\noindent
Moreover, by Green's formula and taking into account \eqref{GaussGreen}, we obtain
\begin{align*}
\displaystyle
   - \frac{N}{ \widetilde{\omega_N}} \int_{B_r} |\nabla  v_N |^{N-2} \nabla  v_N \cdot \nabla_x u \left( \frac{x}{r} \right) \, dx  = -\int_{B_r} u \left( \frac{x}{r} \right)  \, d\mu_N  .\\
\end{align*}

\noindent
Therefore 
\begin{align*}
\displaystyle  
&   \frac{1}{\widetilde{\omega_N}} \int_{\mathbb{R}^N} H_N(u_r,\mu_N) \, dx \\
& = \frac{1}{\widetilde{\omega_N}} \int_{B_1} |\nabla u|^N dx -\int_{B_r} u \left( \frac{x}{r} \right)  \, d\mu_N  + N \ln r - (N-1)\sum_{k=1}^{N-1} \frac{1}{k}  + o(1).
\end{align*}

\noindent
Now, we compute $ \displaystyle \int_{\mathbb{R}^N} u_r \, d\mu_N $.

\begin{align*}
\displaystyle
  \int_{\mathbb{R}^N} u_r \, d\mu_N  = \int_{B_r} u \left( \frac{x}{r} \right) \, d\mu_N  - \int_{B_r} v_N(x) \, d\mu_N  + v_N(r) \int_{B_r} \,  d\mu_N . 
\end{align*}

\begin{align*}
\displaystyle
    v_N(r) \int_{B_r}  \,  d\mu_N   = - \ln V_N - \frac{N^2}{N-1} \ln r + o(1). 
\end{align*}

\begin{align*}
\displaystyle
  - \int_{B_r}  v_N(x) \, d\mu_N  
 = \ln V_N + N^2 \int_0^r \frac{ \rho^{N-1} \ln \left( 1 + \rho^{\frac{N}{N-1}} \right)}{\left( 1 + \rho^{\frac{N}{N-1}}  \right)^N} \, d \rho + o(1).
\end{align*}

\noindent
Taking into account \eqref{contovn} and \eqref{sum}, we have 
\begin{align*}
\displaystyle
&   N^2 \int_0^r \frac{ \rho^{N-1} \ln \left( 1 + \rho^{\frac{N}{N-1}} \right)}{\left( 1 + \rho^{\frac{N}{N-1}}  \right)^N} \, d \rho \\
& = N(N-1) \int_0^{r^{\frac{N}{N-1}}} \frac{1}{\left( 1+t \right)^2} \left( \frac{t}{1+t}  \right)^{N-2} \ln \left( 1+t \right) \, dt \\
& = N\left(   \left[ \left( \frac{t}{1+t}  \right)^{N-1} \ln \left( 1+t \right) \right]_0^{r^{\frac{N}{N-1}}} -  \int_0^{r^{\frac{N}{N-1}}}  \frac{t^{N-1}}{\left( 1+t \right)^N} \, dt   \right) \\
& = N \left( \frac{ r^N- \left( 1 + r^{\frac{N}{N-1}}  \right)^{N-1}    }{\left(1+r^{\frac{N}{N-1}}  \right)^{N-1}} \ln \left( 1+r^{\frac{N}{N-1}} \right)   +     \sum_{k=1}^{N-1} \frac{1}{k} + o(1)  \right) \\
& = N \sum_{k=1}^{N-1} \frac{1}{k} + o(1).\\
\end{align*}

\noindent
Therefore
$$ \displaystyle  \int_{\mathbb{R}^N} u_r \, d\mu_N = \int_{B_r} u \left( \frac{x}{r} \right) \,  d\mu_N  - \frac{N^2}{N-1} \ln r + N \sum_{k=1}^{N-1} \frac{1}{k} + o(1). $$ 

\noindent
Summing up, we have shown that 
\begin{align*}
\displaystyle 
& \inf_{W_{\mu_N}(\R^N)} I  \\
& \leq \frac{1}{\widetilde{\omega_N}} \int_{B_1} |\nabla u|^N  dx + \sum_{k=1}^{N-1} \frac{1}{k} - \ln \left( \frac{1}{V_N} \int_{B_1} e^{u}   dx + N-1 + o(1) \right) + o(1). 
\end{align*} 

\noindent
Passing to the limit as $r \to + \infty$, we obtain 
$$ \displaystyle \inf_{W_{\mu_N}(\R^N)} I  \leq \frac{1}{\widetilde{\omega_N}} \int_{B_1} |\nabla u|^N \, dx + \sum_{k=1}^{N-1} \frac{1}{k} - \ln \left( \frac{1}{V_N} \int_{B_1} e^{u}  \, dx + N-1 \right), $$

\noindent
by which 
$$ \displaystyle \inf_{W_{\mu_N}(\R^N)} I  < J(u) + \sum_{k=1}^{N-1} \frac{1}{k}.$$

\noindent
Finally,
$$ \displaystyle \inf_{W_{\mu_N}(\R^N)} I  \leq  \inf_{W^{1,N}_0(B_1)}  J + \sum_{k=1}^{N-1} \frac{1}{k}.$$
\end{proof}

\noindent
{\bf Acknowledgments} \
The authors are supported by INdAM-GNAMPA and 
they thank PNRR MUR project CN00000013 HUB - National Centre for HPC, Big Data and Quantum Computing
(CUP H93C22000450007). \\
The first and third authors acknowledge financial
support from PRIN PNRR  P2022YFAJH {\sl \lq\lq Linear and Nonlinear PDEs: New directions and applications''} (CUP H53D23008950001). \\
The third author is supported by the INdAM-GNAMPA Projects \lq \lq Fenomeni di concentrazione in PDEs non locali" (CUP E53C22001930001) and \lq \lq Problemi di doppia curvatura su varietà a bordo e legami con le EDP di tipo ellittico'' (CUP E53C23001670001).

\bigskip
\noindent
{\bf Availability of data and material} \ Not applicable.

\end{document}